\documentclass{IEEEtran}
\usepackage{cite}
\usepackage{url}
\usepackage{amsmath,amssymb,amsfonts}
\usepackage{multirow}
\usepackage{amsthm}

\usepackage{graphicx}
\usepackage{textcomp}
\usepackage{epsfig}

\usepackage{epstopdf}
\usepackage[ruled]{algorithm}
\usepackage{algorithmic}
\usepackage{enumerate}
\usepackage{float}
\usepackage{graphics,graphicx}
\usepackage{subfigure}

\usepackage{array}
\usepackage{mathtools}
\usepackage[justification=centering]{caption}
\allowdisplaybreaks[3]
\usepackage{url}
\DeclareGraphicsExtensions{.png}
\graphicspath{{./Pics/}}
\usepackage{color}
\usepackage{xspace}

\usepackage{hyperref}
\hypersetup{
    colorlinks=true,
    linkcolor=blue,
    filecolor=gray,      
    urlcolor=blue,
    citecolor=blue,
}

\def\BibTeX{{\rm B\kern-.05em{\sc i\kern-.025em b}\kern-.08em
		T\kern-.1667em\lower.7ex\hbox{E}\kern-.125emX}}

\newtheorem{lemma}{Lemma}
\newtheorem{theorem}{Theorem}

\newtheorem{assumption}{Assumption}

\usepackage{bbm}
\renewcommand{\theequation}{\arabic{equation}}

\begin{document}
	
	\title{\LARGE \bf Non-Ergodic Convergence Algorithms for Distributed Consensus and Coupling-Constrained Optimization}

\author{Chenyang Qiu, 
Zongli Lin
\thanks{%
\emph{(Corresponding author: Zongli Lin.)}} 
\thanks{%
Chenyang Qiu and Zongli Lin are with the Charles L. Brown Department of Electrical and Computer Engineering, University of Virginia, Charlottesville, VA 22904, USA (e-mail: nzp4an@virginia.edu;  zl5y@virginia.edu).}
} 
\maketitle	
 
\begin{abstract}
We study distributed convex optimization with two ubiquitous forms of coupling: consensus constraints and global affine equalities. We first design a linearized method of multipliers for the consensus optimization problem. Without smoothness or strong convexity, we establish non-ergodic sublinear rates of order $O(1/\sqrt{k})$ for both the objective optimality and the consensus violation. Leveraging duality, we then show that the economic dispatch problem admits a dual consensus formulation, and that applying the same algorithm to the dual economic dispatch yields non-ergodic $O(1/\sqrt{k})$ decay for the error of the summation of the cost over the network and the equality-constraint residual under convexity and Slater’s condition. Numerical results on the IEEE 118-bus system demonstrate faster reduction of both objective error and feasibility error relative to the state-of-the-art baselines, while the dual variables reach network-wide consensus. 
\end{abstract}



\section{Introduction}\label{sec:introduction}
This paper studies large-scale convex optimization problems formulated over networks, which frequently arise in engineering applications. For example, problems such as large-scale machine learning \cite{dean2012large}, distributed control \cite{zhao2015distributed}, and economic dispatch \cite{marzbani2024economic} can be formulated as convex programs in networked systems. Compared with centralized algorithms, distributed algorithms do not rely on coordination from a central node and eliminate potential communication bottlenecks in computational infrastructure, such as high latency or low bandwidth. Furthermore, networks without a central node offer inherent advantages in terms of privacy protection and scalability. A central challenge, however, is to maintain feasibility and reach optimality when each agent observes only local information and communication is limited.

Two broad forms of coupling arise. The first is consensus coupling, which ties agents. Representative approaches include consensus-based subgradient methods \cite{nedic2010constrained}\cite{tsianos2012distributed}, distributed gradient methods \cite{shi2015extra}\cite{xin2018linear}, primal–dual method \cite{bianchi2015coordinate}, and algorithms for composite optimization \cite{wu2022unifying,xu2021distributed,7169615}. These algorithms update local decision variables using (sub)gradients of the node-wise objective functions and aggregate information according to the communication graph, with the goal that, after multiple iterations, all nodes reach consensus while the iterates converge to a stationary point (or a neighborhood thereof).

The other broad form of coupling is the global equality relation binding the agents. A classical equality-coupling instance is economic dispatch problem, whose power-balance constraint couples all generators. For quadratic costs, average-consensus techniques yield explicit distributed solutions \cite{wang2018distributed,wang2024distributed}. For general objective functions, \cite{chang2014multi} provided a class of algorithms that solve the economic dispatch problem with convergence guarantees based on the alternating direction method of multipliers (ADMM). Furthermore, \cite{chen2017admm} proposed an another ADMM-based algorithm which gets rid of the central coordinator.
Ref. \cite{falsone2023augmented} proposed an ADMM-based algorithm, called augmented Lagrangian tracking, for distributed optimization with affine equality and nonlinear inequality coupling constraints.

Based on the duality theory, some algorithms were designed by first formulating the economic dispatch problem's dual problem, which is then solved by consensus optimization algorithms. Ref. \cite{nedic2018improved} proposed the Mirror-P-EXTRA algorithm to solve the economic dispatch problem in a distributed manner and provided the non-ergodic convergence rate of Karush–Kuhn–Tucker (KKT) conditions. In \cite{zhang2020distributed}, a distributed dual gradient-tracking method is proposed; in the absence of local constraints, this method achieves ergodic sublinear convergence of the KKT residual under strongly convex cost functions. Ref. \cite{yang2016distributed} solved the economic dispatch problem over time-varying networks with asymptotic convergence.  





When the objective is nonsmooth and non-strongly convex, many distributed methods provide only asymptotic convergence \cite{falsone2023augmented} or ergodic guarantees for the objective value, i.e., rates for the running average along the local iteration sequence \cite{wu2022unifying,xu2021distributed,zhang2020distributed}. Such ergodic bounds allow pronounced last-iterate oscillations in the objective, which can be undesirable for real-time operation and stability considerations. In addition, several works primarily quantify the decay of KKT/feasibility (consensus) residuals rather than the objective suboptimality itself \cite{7169615,nedic2018improved,zhang2020distributed}. Notably, the
non-ergodic objective error is particularly interpretable in engineering applications, as it directly reflects the \emph{cost gap} at each iteration.

This paper develops a distributed framework that unifies consensus optimization and economic dispatch. For consensus problems, we recast the consensus constraint as a linear equality (equivalently, a graph-Laplacian feasibility condition) and design a linearly augmented Lagrangian with fully distributed updates. Under the assumption of the convexity of the objective function alone, we prove that the sequence generated by the proposed algorithm can gurantee the non-ergodic sublinear convergence rate $O(1/\sqrt{k})$ of the objective error and the consensus error. Leveraging the primal–dual relation between the consensus optimization and economic dispatch problem, we use the proposed algorithm to solve the economic dispatch problem. Under the assumption of the convexity of the objective function and the Slater’s condition, we demonstrate that the sequence generated by the proposed algorithm can guarantee the non-ergodic sublinear convergence rate $O(1/\sqrt{k})$ simultaneously for the global cost and the constraint violation.

The contributions of this paper are summarized as follows. 
\begin{enumerate}
    \item We develop a distributed method for consensus optimization with nonsmooth convex objective functions. With a constant stepsize, we provide the non-ergodic convergence rates for both objective optimality error and consensus error.
    \item Based on the dual decomposition technique, we reformulate economic dispatch as a dual \emph{consensus} problem and adapt the algorithm accordingly. Under the assumption of convexity and Slater's condition, we provide the non-ergodic convergence rates for both objective optimality error and constraint violation.
    \item Numerical experiments conducted on the IEEE 118-bus system demonstrate that our algorithm converges faster than the state-of-the-art algorithms for both objective error and feasibility residuals.
\end{enumerate}

The remainder of this paper is organized as follows. Section \ref{sec: Consensus Optimization} develops an algorithm for consensus optimization and analyzes its convergence. Section \ref{sec: Extension to Economic Dispatch Problem} studies the economic dispatch problem. Section \ref{sec: Simulation} provides numerical experiments to validate our algorithm. Section \ref{sec: Conclusion} concludes the paper.

\textbf{Notation}: 
For a function $g: \mathbb{R}^{m} \rightarrow \mathbb{R}$, its subgradient set at $x \in \mathbb{R}^{m}$ is represented by $\partial g(x)$, i.e., $\tilde{\nabla} g(x) \in \partial g(x)$ is a subgradient of $g$ at $x$. A convex function \( f: X \rightarrow (-\infty, +\infty] \) is said to be proper if \( f(x) > -\infty \) for all \( x \in X \) and $f(x)$ is not trivially equal to $+\infty$. We use the symbols $1_{p}$, $0_{p}$, $O_{p}$, and $I_{p}$ to denote the $p$-dimensional all-one vector, all-zero vector, zero matrix, and identity matrix, respectively. Let $\otimes$ denote the Kronecker product, $\langle\cdot, \cdot\rangle$ denote the Euclidean inner product, and $\|\cdot\|$ denote the $\ell_2$ norm. We write $\Omega \succeq {O}_p$ if it is positive semidefinite and $\Omega \succ {O}_p$ if it is positive definite. For any $\Omega \succeq {O}_p$ and $x \in \mathbb{R}^p$ ,$\|x\|_\Omega^2:=x^{\rm{T}} \Omega x$, $\Omega^{\frac{1}{2}}$ is the square root of $\Omega$ (i.e., $\Omega^{\frac{1}{2}} \Omega^{\frac{1}{2}}=\Omega$ ), 
and $\Omega^{\dagger}$ is the pseudoinverse of $\Omega$. For a set $C \subseteq \mathbb{R}^{p}$, the corresponding indicator function is denoted by $\mathbf{1}_{C}$.

\section{Consensus Optimization}\label{sec: Consensus Optimization}
This section introduces the distributed optimization problem with consensus constraints and formulates it into a linear constrained compact form which can be solved by a linearized method of multipliers.

\subsection{Problem Formulation}
In this paper, we consider a distributed system consisting of $n$ agents that exchange information over a connected undirected communication network. The network is modeled by a graph $(\mathcal{V}, \mathcal{E})$, where $\mathcal{V}=\{1,2, \ldots,n\}$ is the set of nodes and $\mathcal{E} \subseteq\{(i, j) \subseteq \mathcal{V} \times \mathcal{V} \mid i \neq j\}$ is the set of edges. For each agent, $i \in \mathcal{V}$, we denote the set of its neighbors by $\mathcal{N}_i=$ $\{j \in \mathcal{V} \mid (i, j) \in \mathcal{E}\}$. All the agents in this network cooperate to minimize the total cost by solving the following problem,
\begin{align}
    \min_{y_i \in Y_i} \quad & \sum_{i=1}^n g_i(y_i) \label{consensus optimization problem}\\
    \operatorname{s.t.} \quad & y_1 = y_2 = \cdots = y_n , \notag
\end{align} 
where $Y_i \subseteq \mathbb{R}^p$, $y_i$ is the local copy of the decision variable and and $f_i$ is the local objective function owened by each agent $i \in \mathcal{V}$. 
We denote the corresponding Laplacian matrix by $\mathcal{L}$, where 
$$
[\mathcal{L}]_{i j}=\left\{\begin{array}{ll}
\sum_{s \in \mathcal{N}_i} H_{i s}, & i=j, \\
-H_{i j}, & j \in \mathcal{N}_i, \\
0, & \text {otherwise, }
\end{array} \quad i, j \in \mathcal{V},\right.
$$
with $\mathcal{L}_{i j}=\mathcal{L}_{j i}>0$ being the weight of edge $\{i, j\} \in \mathcal{E}$. Since $(\mathcal{V}, \mathcal{E})$ is a connected undirected graph, the null space of $\mathcal{L}$ is $\operatorname{span}\left\{ {1}_n\right\}$. Define $\mathcal{Y} = Y_1 \times Y_2 \times \ldots \times Y_n$, $y\coloneq \left[y_1^{\rm{T}} \; y_2^{\rm{T}} \ldots y_n^{\rm{T}}\right]^{\rm{T}} \in \mathcal{Y}$, $W \coloneq \mathcal{L} \otimes I_{p} \succeq {O}_{np}$. Then, we obtain that the range spaces of the matrices $W$, $W^{\frac{1}{2}}$, $W^{\dagger}$ and $(W^{\dagger})^{\frac{1}{2}}$ are the same and equal to $\left\{y \in \mathbb{R}^{n p} \mid y_1+\cdots+y_n={0}_{p}\right\}$, which is the orthogonal complement of 
$\left\{y \in \mathbb{R}^{n p} \mid y_1 =\cdots=y_n\right\}$.
Hence, the consensus constraint of $y$ is equivalent to $W^{\frac{1}{2}} y = 0_{np}$. The optimal solution of problem \eqref{consensus optimization problem}, denoted by $y^* = \left[(y_1^*)^{\rm{T}} \; (y_2^*)^{\rm{T}} \ldots (y_n^*)^{\rm{T}}\right]^{\rm{T}} \in \mathcal{Y}$ must satisfy $W^{\frac{1}{2}} y^* = 0_{np}$.
Problem \eqref{consensus optimization problem} can therefore be written in the compact form
\begin{align}\label{compactoptimization}
    \min _{y \in \mathcal{Y} } \quad & G(y)=\sum_{i=1}^n g_i(y_i) \\
    \operatorname{s.t.} \quad & W^{\frac{1}{2}} y = {0}_{np}. \notag
\end{align}


\subsection{Algorithm Development} \label{sec: Algorithm Development}
\begin{assumption}\label{ass: G_convex}
For each $i \in \mathcal{V}$, we assume that $Y_i \subseteq \mathbb{R}^p$ is a convex set and $g_i(x_i)$ is a proper and convex function.
\end{assumption}
Based on the above formulation and Assumption \ref{ass: G_convex}, we aim to develop a distributed algorithm for solving problem \eqref{compactoptimization} by utilizing the linearized method of multipliers. To this end, we define the augmented Lagrangian function as
$$L_\rho(y,  v)=G(y)- \left\langle v, W^{\frac{1}{2}} y\right \rangle +\frac{\rho}{2}\left\|W^{\frac{1}{2}} y\right\|^2, $$
where $ v = [v_1^{\rm{T}}\;v_2^{\rm{T}}\ldots v_n^{\rm{T}}]^{\rm{T}} \in \mathbb{R}^{n(d+m)}$ and $\rho>0$. Recall the method of multipliers \cite{boyd2011distributed}, where updates are given by
\begin{subequations}
\begin{alignat}{2}
    y_{k+1} & = \underset{y \in \mathcal{Y}}{\operatorname{argmin}}\left\{ G(y) \! - \!\left\langle  v_k, W^{\frac{1}{2}} y \right\rangle \!+ \!\frac{\rho}{2}\left\| W^{\frac{1}{2}} y \right\|^2 \right\}, \label{admm primal}\\
     v_{k+1} & = v_k-\rho W^{\frac{1}{2}} y_{k+1}. \label{admm dual}
\end{alignat}
\end{subequations}
Since the quadratic term $\frac{\rho}{2}\| W^{\frac{1}{2}} y \|^2$ are not suitable for distributed computation, we linearize it as $\rho \langle W y_k, y \rangle + \frac{\eta}{2} \left\| y - y_k \right\|^2$, where $\eta \geq 0$ is a designed parameter for the algorithm.
Besides, to avoid using $W^{\frac{1}{2}}$, we define
\begin{equation}\label{lambda v}
    \lambda_k = W^{\frac{1}{2}}  v_k
\end{equation}
for distributed implementation, so that $\lambda_{k+1} =\lambda_k-\rho W y_{k+1}$. Therefore, we get the compact form of the proposed iteration for problem \eqref{consensus optimization problem}
\begin{subequations}\label{consensus_algorithm}
    \begin{alignat}{2}
        y_{k+1} =& \, \underset{y \in \mathcal{Y}}{\operatorname{argmin}} \Big\{ G(y) - \langle \lambda_k, y \rangle + \rho \langle W y_k, y \rangle \notag \\
        & \, + \frac{\eta}{2} \left\| y - y_k \right\|^2 \Big \}\label{y_update}\\
        \lambda_{k+1} =& \,\lambda_k-\rho W y_{k+1}.\label{lambda_update}
    \end{alignat}
\end{subequations}

        

The distributed implementation of the above algorithm over the undirected network $(\mathcal{V}, \mathcal{E})$ is detailed in Algorithm \ref{algorithm}. 

\begin{algorithm}[ht]
\caption{}
\label{algorithm}
\begin{algorithmic}[1]
\STATE \textbf{Initialization:} Each agent $i \in \mathcal{V}$, sets $\lambda_{i,0}$ such that $\sum_{i=1}^n \lambda_{i,0} = 0_{p}$ and arbitrarily sets $y_{i,0} \in Y_i$ and. Each agent $i \in \mathcal{V}$ sends the variable $y_{i,0}$ to its neighbors $j\in \mathcal{N}_i$. After receiving the information from its neighbors, each agent $i \in \mathcal{V}$ computes the aggregated information $t_{i,0} = \sum_{j \in \mathcal{N}_i} p_{i j}(y_{i,0}-y_{j,0})$. 
\FOR{$k = 0,1,2,\ldots, $} 
    \STATE Each agent $i \in \mathcal{V}$ updates the variable $y_{i,k+1} = {\operatorname{argmin}}_{y_i \in Y_i} $ $\{ g_i(y_i) - \langle \lambda_{i,k} - \rho t_{i,k}, y_i \rangle + \frac{\eta}{2} \left\| y_i - y_{i,k} \right\|^2 \}$.
    \STATE Each agent $i \in \mathcal{V}$ sends the variable $y_{i,k+1}$ to its neighbors $j\in \mathcal{N}_i$. After receiving the information from its neighbors, each agent $i \in \mathcal{V}$ computes the aggregated information $t_{i,k+1} = \sum_{j \in \mathcal{N}_i} p_{i j}(y_{i,k+1}-y_{j,k+1})$. 
    \STATE Each agent $i \in \mathcal{V}$ updates $\lambda_{i,k+1}= \lambda_{i,k}-\beta_k t_{i,k+1}$.
\ENDFOR
\end{algorithmic}
\end{algorithm}

\subsection{Convergence Analysis}\label{sec: Convergence Analysis}
We denote the variables generated by the proposed algorithm at the $k$th iteration by $z_k = [y_k^{\rm{T}}\; \lambda_k^{\rm{T}}]^{\rm{T}}$ and the difference between the $(k+1)$th iteration and the $k$th iteration by $\Delta y_{k+1} = y_k - y_{k+1}$, $\Delta \lambda_{k+1} = \lambda_k - \lambda_{k+1}$, and $\Delta z_{k+1} = z_k - z_{k+1}$. 

\begin{lemma}\label{lem: 1}
Suppose Assumption \ref{ass: G_convex} holds. If $\eta I - \rho W > 0$, then the sequences $y_k$ and $\lambda_k$ generated by the proposed algorithm satisfy the following inequality,
    \begin{align} \label{Delta y_k+1leq}
        &\, \frac{1}{k}\sum_{s=0}^k (\| \Delta y_{s+1} \|_{\Omega}^2 + \| \Delta \lambda_{s+1} \|_{\Omega}^2 )\notag \\
        \leq &\, \frac{1}{k}(  \|y_0 - y^*\|_{\Omega}^2 + \frac{1}{\rho} \|\lambda_0 - \lambda^*\|_{W^{\dagger}}^2 ),
    \end{align}
where $\Omega = \eta I - \rho W $.  
\end{lemma}
\begin{proof}
    According to \eqref{y_update}, we have 
    \begin{align}
    &  \lambda_{k} - \rho W y_k - \eta(y_{k+1} - y_k) \notag \\ 
    \in & \partial G(y_{k+1}) + \partial \mathbf{1}_{\mathcal{Y}}(y_{k+1}),
    \end{align}
    where $\mathbf{1}_{\mathcal{Y}}$ is the indicator function of $\mathcal{Y}$, $\partial G(y_{k+1})$ and $\partial \mathbf{1}_{\mathcal{Y}}(y_{k+1})$ are the sets of subgradients of $G$ and $\mathbf{1}_{\mathcal{Y}}$ at $y_{k+1}$, respectively. 
    Based on KKT conditions \cite{boyd2004convex}, we have that any vector $\lambda^*$ satisfying the following equation is a dual optimal point corresponding to problem \eqref{consensus optimization problem}
    \begin{equation}\label{lambda*}
        \lambda^* \in \partial G(y^*) + \partial \mathbf{1}_{\mathcal{Y}}(y^*).
    \end{equation}
    Therefore, plus with \eqref{lambda_update}, for any $\tilde{\nabla} G(y_{k+1}) \in \partial G(y_{k+1})$ and $\tilde{\nabla} \mathbf{1}_{\mathcal{Y}}(y_{k+1}) \in \partial \mathbf{1}_{\mathcal{Y}}(y_{k+1})$, we have
    \begin{align}\label{constraits_error}
    &\, \tilde{\nabla} \mathbf{1}_{\mathcal{Y}}(y_{k+1}) - \tilde{\nabla} \mathbf{1}_{\mathcal{Y}}(y^*)+ \Omega (y_{k+1} - y_k)- (\lambda_{k+1} - \lambda^*)\notag \\ 
    = & -\tilde{\nabla} G(y_{k+1}) + \tilde{\nabla} G(y^*) , 
    \end{align}
    where $\Omega = \eta I_{n(d+m)} - \rho W \succ O_{n(d+m)}$. By the convexity of $G$, we have
    \begin{align}\label{norm_geq0}
        0
        \leq &\ \langle \tilde{\nabla} G(y_{k+1}) - \tilde{\nabla} G(y^*), y_{k+1} -y^*\rangle  \notag \displaybreak[0]\\
        = &\, -\langle \tilde{\nabla} \mathbf{1}_{\mathcal{Y}}(y_{k+1}) - \tilde{\nabla} \mathbf{1}_{\mathcal{Y}}(y^*) , y_{k+1} -y^*\rangle \notag \\
        &\, - \langle \Omega (y_{k+1} - y_k), y_{k+1} -y^* \rangle \notag \\
        &\, + \langle \lambda_{k+1} - \lambda^*, y_{k+1} -y^* \rangle \notag \displaybreak[0]\\
        \leq &\ - \langle \Omega (y_{k+1} - y_k), y_{k+1} -y^* \rangle \notag \\
        &\, + \langle \lambda_{k+1} - \lambda^*, y_{k+1} -y^* \rangle \notag \\
        = &\, - \langle \Omega (y_{k+1} - y_k), y_{k+1} -y^* \rangle \notag \\
        &\, - \left\langle \lambda_{k+1} - \lambda^*, \frac{W^{\dagger}}{\rho}(\lambda_{k+1} - \lambda_{k}) \right\rangle \notag \\
        = & \frac{1}{2} \Big( \|y_k - y^*\|_{\Omega}^2- \|y_{k+1} - y_k\|_{\Omega}^2 -\|y_{k+1} - y^*\|_{\Omega}^2\Big) \notag \\
        &\, + \frac{1}{2 \rho} \Big(\|\lambda_k - \lambda^*\|_{W^{\dagger}}^2 - \|\lambda_{k+1} - \lambda_k\|_{W^{\dagger}}^2  \notag \\
        &\, -\|\lambda_{k+1} - \lambda^*\|_{W^{\dagger}}^2\Big),
    \end{align}
    which indicates that $\forall k>1$, $\|y_k - y^*\|_{\Omega}^2 + \|\lambda_k - \lambda^*\|_{W^{\dagger}}^2$ is nonincreasing, i.e.,  
    \begin{align}\label{nonincreasing}
         &\, \|y_k - y^*\|_{\Omega}^2 + \|\lambda_k - \lambda^*\|_{W^{\dagger}}^2 \notag \\
         \geq &\, \|y_{k+1} - y^*\|_{\Omega}^2 + \frac{1}{ \rho} \|\lambda_{k+1} - \lambda^*\|_{W^{\dagger}}^2.
    \end{align}
    Sum \eqref{norm_geq0} from $s=1$ to $s=k$ yields
    \begin{align}
        &\, \frac{1}{k}\sum_{s=0}^k (\| \Delta y_{s+1} \|_{\Omega}^2 + \| \Delta \lambda_{s+1} \|_{\Omega}^2 )\notag \\
        \leq &\, \frac{1}{k}(  \|y_0 - y^*\|_{\Omega}^2 + \frac{1}{\rho} \|\lambda_0 - \lambda^*\|_{W^{\dagger}}^2 ),
    \end{align}
    and thus the proof of Lemma \ref{lem: 1} is completed.
\end{proof}
\begin{lemma}\label{lem: 2}
Suppose the conditions of Lemma \ref{lem: 1} hold. For $k>0$,
\begin{equation}\label{rate_Delta_y_k}
    \|\Delta z_{k+1}\|_{\hat{\Omega}} \leq \frac{1}{\sqrt{k}}(  \|z_0 - z^*\|_{\hat{\Omega}} ),
    \end{equation}
where $\hat{\Omega}= \begin{bmatrix}
    \Omega & O_{np} \\
    O_{np} & \tfrac{1}{\rho} W^{\dagger}
\end{bmatrix} $. 
\end{lemma}
\begin{proof}   
    Applying \eqref{constraits_error} for $k+1$ and $k$, we have 
    \begin{align}
        & -\Omega (\Delta y_{k+1} - \Delta y_k) + \Delta \lambda_{k+1} \notag \\
        = & \tilde{\nabla} G(y_{k+1}) - \tilde{\nabla} G(y_k) + \tilde{\nabla} \mathbf{1}_{\mathcal{Y}} (y_{k+1}) - \tilde{\nabla} \mathbf{1}_{\mathcal{Y}} (y_{k})
    \end{align}
     By \eqref{lambda_update} and the convexity of $G$ and $\mathbf{1}_{\mathcal{Y}}$, we have
    \begin{align}
         0 
        \leq &\, \langle \tilde{\nabla} G(y_{k+1}) - \tilde{\nabla} G(y_k) + \tilde{\nabla} \mathbf{1}_{\mathcal{Y}} (y_{k+1}) - \tilde{\nabla} \mathbf{1}_{\mathcal{Y}} (y_{k}), \notag \\  
        & \, y_{k+1} - y_{k} \rangle \notag \\
        = &\, \langle -\Omega (\Delta y_{k+1} - \Delta y_k) + \Delta \lambda_{k+1}, \Delta y_{k+1} \rangle \notag \\
        = &\, \langle -\Omega (\Delta y_{k+1} - \Delta y_k), \Delta y_{k+1} \rangle  \notag \\
        &\, - \frac{1}{\rho} \langle \Delta \lambda_{k+1},  W^{\dagger}(\Delta \lambda_{k+1} - \Delta \lambda_{k} \rangle \notag \\
        = &\, - \frac{1}{2}(\| \Delta y_{k+1} \|_{\Omega}^2 - \| \Delta y_k \|_{\Omega}^2 - \| \Delta y_{k+1} - \Delta y_{k}\|_{\Omega}^2) \notag \\
        &\, - \frac{1}{2 \rho}(\| \Delta \lambda_{k+1} \|_{W^{\dagger}}^2 - \| \Delta \lambda_k \|_{W^{\dagger}}^2 \notag \\
        &\,- \| \Delta \lambda_{k+1} - \Delta \lambda_{k}\|_{W^{\dagger}}^2),
    \end{align}
    which indicates that $\Delta z_{k+1}$ is nonincreasing, i.e.,
    \begin{align}\label{nonincrease_delta_z}
        \|\Delta z_{k+1}\|_{\hat{\Omega}}^2 \leq \|\Delta z_{k}\|_{\hat{\Omega}}^2.
    \end{align}
    Combining \eqref{Delta y_k+1leq} and \eqref{nonincrease_delta_z} yields \eqref{rate_Delta_y_k} in Lemma \ref{lem: 2}.    
\end{proof}

To characterize the convergence rates of the optimality gap and the consensus error, we establish the following theorem.

\begin{theorem}
    Suppose the conditions of Lemma \ref{lem: 1} hold, then the sequence $z_k$ generated by the proposed algorithm satisfies 
    \begin{equation}\label{theorem1}
    \begin{aligned}
        &\ \frac{1}{\sqrt{k}} (\frac{1}{2}\| z_0 - z^* \|^2_{\hat{\Omega}} + \|\tilde{\nabla} G(y^*) \| \| z_0 - z^* \|_{\hat{\Omega}})  \\
        \geq &\ G(y_{k+1}) - G(y^*) \\
        \geq &\, - \frac{1}{\sqrt{k}}\|\tilde{\nabla} G(y^*)\| \| z_0 - z^* \|_{\hat{\Omega}},
    \end{aligned}
    \end{equation}
    and 
    \begin{equation}
        \| y_{k+1} \|_{W} = \frac{1}{\sqrt{ k}}(  \|z_0 - z^*\|_{\hat{\Omega}} ).
    \end{equation}
\end{theorem}
\begin{proof}
    By the convexity of $G$, we have
    \begin{align}\label{convexity_G}
        & \, G(y_{k+1}) - G(y^*) \notag \\
        \leq &\, \langle \tilde{\nabla} G(y_{k+1}) - \tilde{\nabla} G(y^*), y_{k+1} - y^* \rangle +\langle \tilde{\nabla} G(y^*), y_{k+1} - y^* \rangle 
    \end{align}
    According to \eqref{norm_geq0}, the first term of the left-hand side of \eqref{convexity_G} is bounded by the following inequalities
    \begin{align}\label{nablaG(y_k+1)-G(y^*)y_k+1-y*}
        &\ \langle \tilde{\nabla} G(y_{k+1}) - \tilde{\nabla} G(y^*), y_{k+1} - y^* \rangle \notag \\
        \leq &\ \frac{1}{2} \Big( \|y_k - y^*\|_{\Omega}^2- \|y_{k+1} - y_k\|_{\Omega}^2 -\|y_{k+1} - y^*\|_{\Omega}^2\Big) \notag \\
        &\, + \frac{1}{2 \rho} \Big(\|\lambda_k - \lambda^*\|_{W^{\dagger}}^2 - \|\lambda_{k+1} - \lambda_k\|_{W^{\dagger}}^2  \notag \\
        &\, -\|\lambda_{k+1} - \lambda^*\|_{W^{\dagger}}^2\Big) \notag \\
        = &\ \frac{1}{2} \Big( \langle \Omega (y_k - y_{k+1} ), y_k + y_{k+1} - y^* \rangle - \|y_{k+1} - y_k\|_{\Omega}^2 \Big) \notag \\
        &\, + \frac{1}{2 \rho} \Big(\langle W^{\dagger}(\lambda_k - \lambda_{k+1}),  \lambda_k + \lambda_{k+1} - 2 \lambda^* \rangle   \notag \\
        &\, - \|\lambda_{k+1} - \lambda_k\|_{W^{\dagger}}^2 \Big) \notag \\
        \leq &\ \frac{1}{4} \Big( \| y_k - y_{k+1} \|_{\Omega} (\| y_k - y^* \|_{\Omega} + \| y_{k+1} - y^* \|_{\Omega})  \notag \\
        &\, - \|y_{k+1} - y_k\|_{\Omega}^2 \Big) \notag \\
        &\, + \frac{1}{4 \rho} \Big( \| \lambda_k - \lambda_{k+1}\|_{W^{\dagger}}( \| \lambda_k - \lambda^*\|_{W^{\dagger}} + \| \lambda_{k+1} - \lambda^* \|_{W^{\dagger}} )   \notag \\
        &\, - \|\lambda_{k+1} - \lambda_k\|_{W^{\dagger}}^2 \Big) \notag \\
        \leq &\ \frac{\| z_0 - z^* \|_{\hat{\Omega}}}{4 \sqrt{k}} \Big( (\| y_k - y^* \|_{\Omega} + \| y_{k+1} - y^* \|_{\Omega} ) \notag \\
        &\, + ( \| \lambda_k - \lambda^*\|_{W^{\dagger}} + \| \lambda_{k+1} - \lambda^* \|_{W^{\dagger}} ) \Big)  \notag \\
        \leq &\ \frac{\| z_0 - z^* \|^2_{\hat{\Omega}}}{2 \sqrt{k}}.
   \end{align}
   As for the second term of the right-hand side of \eqref{convexity_G}, we have
    \begin{align}\label{nablaG(y^*)y_k+1-y*}
        \langle \tilde{\nabla} G(y^*), y_{k+1} - y^* \rangle 
        \leq &\, \langle \tilde{\nabla} G(y^*), y_{k+1}  \rangle \notag \\
        \leq &\, \frac{1}{\rho} \|\tilde{\nabla} G(y^*)\| \| \lambda_{k+1} - \lambda_{k} \|_{W^{\dagger}} \notag \\
        \leq &\, \frac{\|\tilde{\nabla} G(y^*)\|}{\sqrt{ k }}( \|z_0 - z^*\|_{\hat{\Omega}} ).
    \end{align}
    where the first inequality is due to \eqref{lambda_update} and the fact that $y^*$ is in the null space of $\lambda_{k+1}$. Combining \eqref{nablaG(y_k+1)-G(y^*)y_k+1-y*} and \eqref{nablaG(y^*)y_k+1-y*} yields the first inequality of \eqref{theorem1}.    
    Similarly, for the lower bound of the optimality gap in \eqref{theorem1},  
    \begin{align}
        G(y_{k+1}) - G(y^*) 
        \geq &\, \langle \tilde{\nabla} G(y^*), y_{k+1} - y^* \rangle \notag \\
        \geq &\, -\| \tilde{\nabla} G(y^*) \| \| \lambda_{k+1} - \lambda_{k} \|_{W^{\dagger}} \notag \\
        \geq &\, -\frac{\|\tilde{\nabla} G(y^*)\|}{\sqrt{ k }}( \|z_0 - z^*\|_{\hat{\Omega}} ).
    \end{align}    
\end{proof}

\section{Extension to Economic Dispatch}\label{sec: Extension to Economic Dispatch Problem}
In this section, we study the economic dispatch problem, establish its primal--dual correspondence with consensus optimization, and, leveraging this link, adapt Algorithm~\ref{algorithm} to it and provide the corresponding convergence analysis.

\subsection{Problem Formulation}
The objective of the economic dispatch problem is to minimize the total cost of power generation in the network while meeting the total power demand and satisfying individual constraints. For each generator $i \in \mathcal{V}$, let $x_i$ be its power output and $ f_i (x_i) $ be its local cost function. Let $d$ represent the total demand that must be satisfied by the power output of all generators. Initially, each generator $i$ is assigned a virtual local demand $d_i$, with $\sum_{i=1}^n d_i = \Bar{d}$. Then, the economic dispatch problem is formulated as a constrained optimization problem, which has the following general form,
\begin{align}\label{economic dispatch problem}
 ~ \min _{x}& ~ f(x) = \sum_{i=1}^n f_i(x_i), \\
\operatorname{s.t.} & ~ \sum_{i=1}^n x_i = \sum_{i=1}^n d_i = \Bar{d},  \quad x\in X, \notag 
\end{align}
where $x = [x_1^{\rm{T}} \;x_2^{\rm{T}} \ldots x_n^{\rm{T}}]^{\rm{T}} \in \mathbb{R}^{n p}$, $X_i \subseteq \mathbb{R}^p$ is the local constraint set for each $x_i$, and $X = X_1 \times X_2 \times \ldots \times X_n$. Note that $p=1$ is usually set in the field of power systems. However, we will address the constrained optimization problem (\ref{economic dispatch problem}) where $p$ is allowed to be any finite positive integer. 
\begin{assumption}\label{ass: convex}
For each $i \in \mathcal{V}$, we assume that $f_i(x_i): \mathbb{R}^{p} \rightarrow (-\infty, +\infty]$ is a proper and convex function.
\end{assumption}

\begin{assumption}\label{ass: slater's condition}
    The set $X_i$ is nonempty, closed, and convex. Furthermore, the Slater's condition is satisfied, i.e., the constraint $\sum_{i=1}^n x_i-d_i = 0$ is satisfied for at least one point in the relative interior of $X$.
\end{assumption}

\subsection{Primal-Dual Relationship}
By introducing the Lagrangian multiplier $\delta \in \mathbb{R}^p$, the Lagrangian function associated with problem \eqref{economic dispatch problem} is defined as 
$L( x, \delta) = f( x) + \left\langle \delta , \sum_{i=1}^n (x_i - d_i) \right \rangle$. The dual problem to problem \eqref{economic dispatch problem} is then given as
\begin{align} \label{dual}
    \underset{\delta \in \mathbb{R}^p}{\max} \ \underset{ x }{\min } L( x, \delta) 
    = & \ \underset{\delta \in \mathbb{R}^p}{\max} \ \sum_{i=1}^n \underset{x_i \in X_i }{\min }  \left\{ f_i(x_i) + \langle \delta, x_i -d_i  \rangle \right\}  \notag\\
    = & \ \underset{\delta \in \mathbb{R}^p}{\max} \sum_{i=1}^n - f_i^*(-\delta) - \langle \delta, d_i \rangle ,
\end{align}    
where for each $i \in \mathcal{V}$, the convex conjugate function $f_i^*(\delta)$ is defined as
\begin{align*}
     f_i^*(\delta)
    = & \, \underset{x_i\in X_i}{\max} \left\{\langle \delta,x_i \rangle  - f_i(x_i) \right\} \\
    = & \,  -\underset{x_i\in X_i}{\min} \left\{f_i(x_i) - \langle \delta,x_i \rangle \right\}.
\end{align*}
Let $g_i(\delta) = f^*_i(-\delta) + \langle \delta, d_i \rangle $. Then, the dual problem in (\ref{dual}) becomes 
$$\min _{\delta \in \mathbb{R}^p} \sum_{i=1}^n g_i(\delta).$$
Assign a local copy of the dual variable $y_i$ to each agent $i \in \mathcal{V}$, we have the consensus optimization problem 
\begin{align}
    \min_{y \in \mathbb{R}^{np}} \quad & G(y) = \sum_{i=1}^n g_i(y_i) \label{consensus dual optimization problem}\\
    \operatorname{s.t.} \quad & y_1 = y_2 = \cdots = y_n , \notag
\end{align} 
which is equivalent to \eqref{consensus optimization problem}.

Denote a optimal solution of problem \eqref{economic dispatch problem} by $ x^* = [(x_1^*)^{\rm{T}}\;(x_2^*)^{\rm{T}} \ldots (x_n^*)^{\rm{T}}]^{\rm{T}}$. Then, $ x^*$ and $y^*$ as the optimal pair of the Lagrangian function $L(x,y)$ if and only if
\begin{enumerate}
    \item $ x^*$ is the optimal solution to problem \eqref{economic dispatch problem}, i.e., $f( x^*) \leq f(x) $ for any $ x \in X$ and $\sum_{i=1}^n x_i^* = d $.
    \item $y^*$ is an optimal solution to the dual problem \eqref{consensus dual optimization problem}.
\end{enumerate}
Assumptions \ref{ass: convex} and \ref{ass: slater's condition} guarantee the strong duality holds \cite{bertsekas1997nonlinear}, i.e., $G(y^*) = -f(x^*)$. 
\subsection{Algorithm Development}
With the form of the consensus optimization problem, we can use iteration \eqref{consensus_algorithm} to solve the dual problem of \eqref{economic dispatch problem}. Define $d = [d_1^{\rm{T}}\; d_2^{\rm{T}} \ldots d_n^{\rm{T}}]^{\rm{T}}$. Applying the definition of $G$ into \eqref{consensus_algorithm} gives the following iteration,
\begin{align}
    & \, y_{k+1} \notag \\
    =& \, \underset{y \in \mathbb{R}^{np}}{\operatorname{argmin}} \Big\{ G(y) - \langle \lambda_k, y \rangle + \rho \langle W y_k, y \rangle + \frac{\eta}{2} \left\| y - y_k \right\|^2 \Big\} \notag \\
    =& \, \underset{y \in \mathbb{R}^{np}}{\operatorname{argmin}} \max _{x \in X} \Big\{ -\left( f(x) - \langle  x-d, y \rangle \right) - \langle \lambda_k - \rho W y_k, y \rangle \notag \\
    & \, \left. + \frac{\eta}{2} \left\| y - y_k \right\|^2\right\} \notag \\
    =& \, y_k + \frac{1}{\eta} \left( \tilde{x}-d + \lambda_k - \rho W y_k\right), \label{y_k+1_1}
\end{align}
where $\tilde{x}$ is the optimal value solved by
\begin{align}\label{tildex_k+1}
    \tilde{x} = &\, \max _{x \in X} \Big\{ -\left( f(x) - \langle x-d, y_{k+1} \rangle \right) - \langle \lambda_k - \rho W y_k, y_{k+1} \rangle \notag \\
    & \, \left. + \frac{\eta}{2} \left\| y_{k+1} - y_k \right\|^2\right\} 
\end{align}
Plugging \eqref{y_k+1_1} into \eqref{tildex_k+1} yields 
\begin{subequations}\label{}
    \begin{alignat}{2}
        &\, x_{k+1} \notag \\
        = & \, \underset{x \in X}{\operatorname{argmin}} \left \{f(x) + \frac{\eta}{2} \|  y_k + \frac{1}{\eta} \left(x-d + \lambda_k - \rho W y_k\right)\|^2 \right \}, \label{x_k+1} \\
        & \, y_{k+1} \notag \\
        =& \, y_k + \frac{1}{\eta} \left(x_{k+1} - d + \lambda_k - \rho W y_k\right), \label{y_k+1}
    \end{alignat}
\end{subequations}
Since the objective function of the above subproblem is separable, the update of $x_{k+1}$, $y_{k+1}$, and $\lambda_{k+1}$ can be carried out in a distributed manner, which is described in Algorithm \ref{algorithm_coupling}, where for each $i \in \mathcal{V}$, $J_{i,k}(\xi) = y_{i,k} + \frac{1}{\eta} \left(\xi - d_i + \lambda_{i,k} - t_{i,k}\right) $, and $t_{i,k} = \sum_{j \in \mathcal{N}_i} p_{i j}(y_{i,k}-y_{j,k})$.

\begin{algorithm}[ht]
\caption{}
\label{algorithm_coupling}
\begin{algorithmic}[1]
\STATE \textbf{Initialization:} Each agent $i \in \mathcal{V}$, sets $\lambda_{i,0}$ such that $\sum_{i=1}^n \lambda_{i,0} = 0_{p}$ and arbitrarily sets $y_{i,0} \in \mathbb{R}^p$ and. Each agent $i \in \mathcal{V}$ sends the variable $y_{i,0}$ to its neighbors $j\in \mathcal{N}_i$. After receiving the information from its neighbors, each agent $i \in \mathcal{V}$ computes the aggregated information $t_{i,0} = \sum_{j \in \mathcal{N}_i} p_{i j}(y_{i,0}-y_{j,0})$. 
\FOR{$k = 0,1,\ldots, $}    
    \STATE Each agent $i \in \mathcal{V}$ updates the variable $x_{i,k+1}$ by $\underset{\xi \in X_i}{\operatorname{argmin}} \{f_i(\xi) + \frac{\eta}{2}\|J_{i,k}(\xi)\|^2 \}$.
    \STATE Each agent $i \in \mathcal{V}$ updates the variable $y_{i,k+1} = J_{i,k}(x_{i,k+1})$ and then sends the variable $y_{i,k+1}$ to its neighbors $j\in \mathcal{N}_i$..
    \STATE After receiving the information from its neighbors, each agent $i \in \mathcal{V}$ computes the aggregated information $t_{i,k} = \sum_{j \in \mathcal{N}_i} p_{i j}(y_{i,k}-y_{j,k})$, and further update$\lambda_{i,k+1}= \lambda_{i,k}-\beta_k t_{i,k+1}$.
\ENDFOR
\end{algorithmic}
\end{algorithm}

\subsection{Convergence Analysis}
In this section, we present the convergence analysis for algorithm \ref{algorithm_coupling}. Specifically, we provide the convergence rates of the optimality gap in the function value and the feasibility error.

\begin{theorem}\label{theo: primal gap}
Consider the distributed algorithm in Algorithm \ref{algorithm_coupling} under Assumptions \ref{ass: convex} and \ref{ass: slater's condition}. If the parameters $\eta$ and $\rho$ of Algorithm \ref{algorithm} satisfies $\eta I - \rho W \succ O_{n(d+m)}$, then the difference between the summation of the cost over the network and the optimal cost at the $(k+1)$th iteration is bounded as follows,
    \begin{equation}\label{primal gap}
    \begin{aligned}
        &\ \frac{1}{ \sqrt{k}} (\frac{1}{2}\| z_0 - z^* \|^2_{\hat{\Omega}} + \|y^*\|_{\Omega}\| z_0 - z^* \|_{\hat{\Omega}})  \\
        \geq &\ f(x_{k+1}) - f(x^*) \\
        \geq &\, - \frac{1}{\sqrt{k}} \|y^*\|_{\Omega}\| z_0 - z^* \|_{\hat{\Omega}}.
    \end{aligned}
    \end{equation}
    
\end{theorem}
\begin{proof}
    Taking the derivative of the function of the subproblem \eqref{x_k+1} gives
    \begin{align*}
        & y_k + \frac{1}{\eta} \left(x_{k+1}-d + \lambda_k - \rho W y_k\right) \\
        \in &\, - (\partial f(x_{k+1})+ \partial \mathbf{1}_{\mathcal{X}}(x_{k+1})),
    \end{align*}
    substitution of which in \eqref{y_k+1} gives
    $$
    y_{k+1} \in - (\partial f(x_{k+1})+ \partial \mathbf{1}_{\mathcal{X}}(x_{k+1})),
    $$
    and similarly,
    \begin{align*}
        y^* \in - (\partial f(x^*) + \partial \mathbf{1}_{\mathcal{X}}(x^*)).
    \end{align*}
    By the convexity of $f$, we have
    \begin{align}\label{fxk+1-fx*}
        &\ f(x_{k+1}) - f(x^*) \notag \\
        \leq &\, \langle x^* - x_{k+1}, y_{k+1} - y^* \rangle + \langle x^* - x_{k+1}, y^* \rangle.
    \end{align}
    Based on the definition of $g_i$ and the conjugate duality,
    \begin{align*}
        -x_{i,k+1} + d_i 
        \in &\ \partial g_i(y_{i,k+1}) \\
        = &\ \partial f_i^*(-y_{i,k+1}) + d_i  \\
        = &\, -\arg\min_{x \in X_i} \left\{ f_i(x) + \langle x, y_{i,k+1} \rangle \right\} + d_i, 
    \end{align*}
    and similarly,
    \begin{align*}
        -x^* + d_i \in \partial g_i(y^*).
    \end{align*}
    Therefore, by substituting the subgradients of $G$ at $y_{i,k+1}$ and $y^*$ with $-x_{i,k+1} + d_i$ and $-x^* + d_i$ in \eqref{nablaG(y_k+1)-G(y^*)y_k+1-y*}, respectively, the first term of the right-hand side of \eqref{fxk+1-fx*} is upper bounded as follows,
    \begin{align}\label{norm_inequ}
        &\ \langle x^* - x_{k+1}, y_{k+1} -y^*\rangle  \notag \displaybreak[0]\\
        \leq &\ \frac{1}{2} \Big( \|y_k - y^*\|_{\Omega}^2- \|y_{k+1} - y_k\|_{\Omega}^2 -\|y_{k+1} - y^*\|_{\Omega}^2\Big) \notag \\
        &\, + \frac{1}{2 \rho} \Big(\|\lambda_k - \lambda^*\|_{W^{\dagger}}^2 - \|\lambda_{k+1} - \lambda_k\|_{W^{\dagger}}^2  \notag \\
        &\, -\|\lambda_{k+1} - \lambda^*\|_{W^{\dagger}}^2\Big) \notag \\
        = &\ \frac{1}{2} \Big( \langle \Omega (y_k - y_{k+1} ), y_k + y_{k+1} - y^* \rangle - \|y_{k+1} - y_k\|_{\Omega}^2 \Big) \notag \\
        &\, + \frac{1}{2 \rho} \Big(\langle W^{\dagger}(\lambda_k - \lambda_{k+1}),  \lambda_k + \lambda_{k+1} - 2 \lambda^* \rangle   \notag \\
        &\, - \|\lambda_{k+1} - \lambda_k\|_{W^{\dagger}}^2 \Big) \notag \\
        \leq &\ \frac{1}{4} \Big( \| y_k - y_{k+1} \|_{\Omega} (\| y_k - y^* \|_{\Omega} + \| y_{k+1} - y^* \|_{\Omega} ) \notag \\
        &\, - \|y_{k+1} - y_k\|_{\Omega}^2 \Big) \notag \\
        &\, + \frac{1}{4 \rho} \Big( \| \lambda_k - \lambda_{k+1}\|_{W^{\dagger}}( \| \lambda_k - \lambda^*\|_{W^{\dagger}} + \| \lambda_{k+1} - \lambda^* \|_{W^{\dagger}} )   \notag \\
        &\, - \|\lambda_{k+1} - \lambda_k\|_{W^{\dagger}}^2 \Big) \notag \\ 
        \leq &\ \frac{\| z_0 - z^* \|^2_{\hat{\Omega}}}{2 \sqrt{k}}.
    \end{align}
    Besides, since for any $k \geq 0$, $\lambda^*$ and $\lambda_{k+1}$ are in the range space of $W$ while $y^*$ is in the null space of $W$, we have
    \begin{equation}\label{lambda_k+1-lambda^*y^*}
        \langle \lambda_{k+1} - \lambda^*, y^* \rangle = 0.
    \end{equation}
    Based on \eqref{lambda_k+1-lambda^*y^*}, we have the following inequality 
    \begin{align}\label{nabla G(y_{k+1})-nabla G(y^*) y^*}
        &\ \langle h(x^*) - h(x_{k+1}), y^* \rangle \notag \\
        = &\ \big( \eta \langle -\epsilon_{k+1}, y^* \rangle + \langle \Omega(y_k - y_{k+1}), y^* \rangle + \langle \lambda_{k+1} - \lambda^*, y^* \rangle \big) \notag \\
        \leq &\ \langle \Omega y_k - y_{k+1}, y^* \rangle \notag \\
        = &\ \| y^* \|_{\Omega} \| y_k - y_{k+1} \|_{\Omega}.
    \end{align}
    Combining \eqref{norm_inequ} and \eqref{nabla G(y_{k+1})-nabla G(y^*) y^*} yields the first inequality of \eqref{primal gap}.
    In the end, we get the second inequality of \eqref{primal gap} by      
    \begin{align}\label{fxk+1-fx*1}
        &\ f(x_{k+1}) - f(x^*) \notag \\
        \geq &\, \langle h(x^*) - h(x_{k+1}), y^*\rangle \notag\\
        \geq &\, - \| y^* \|_{\Omega} \| y_k - y_{k+1} \|_{\Omega} .
    \end{align}
\end{proof}
\begin{theorem}\label{theo: constraints}
Under the same condition of Theorem \ref{theo: primal gap}, we have the convergence rate of the feasibility error as follows,
    \begin{equation}
        \left \| \sum_{i=1}^n x_{i,k+1} - d_i \right\| \leq \frac{1}{\sqrt{k}} \| z_0 - z^* \|_{\hat{\Omega}}.
    \end{equation}
\end{theorem}
\begin{proof}
    \begin{align}\label{violation_in_proof}
        &\, (1_n^{\rm{T}} \otimes I_{d+m}) (x_{k+1} - d) \notag \\
        = &\, (1_n^{\rm{T}} \otimes I_{d+m}) \big( \Omega(y_{k+1} - y_k) - (\lambda_{k+1} - \lambda^*) \big) \notag \\
        = &\, (1_n^{\rm{T}} \otimes I_{d+m}) \Omega (y_{k+1} - y_k) ,
    \end{align}
    where the last equality is due to that $\lambda_{k+1}$ and $\lambda^*$ are in the range space of $W$.
    Therefore, we have 
    \begin{align}
        &\, \left \| \sum_{i=1}^n x_{i,k+1} - d_i \right\| \notag \\
    = &\ \| (1_n^{\rm{T}} \otimes I_{d+m}) (x_{k+1} - d) \| \notag \\
    \leq &\ \|(1_n^{\rm{T}} \otimes I_{d+m})\| \| x_{k+1} - d \| \notag \\
    \leq &\ \|(1_n^{\rm{T}} \otimes I_{d+m})\| \| \Omega\|^{\frac{1}{2}} \| y_{k+1} - y_k \|_{\Omega} \notag \\
    \leq &\ (n \| \Omega\|)^{\frac{1}{2}} \| y_{k+1} - y_k \|_{\Omega} \notag \\
    \leq  &\ \frac{1}{\sqrt{k}} \| z_0 - z^* \|_{\hat{\Omega}} ,
    \end{align}
    which completes the proof of Theorem \ref{theo: constraints}.
\end{proof}
\section{Numerical Experiments} \label{sec: Simulation}

In this section, we conduct several numerical experiments to validate the proposed algorithm and compare it with three existing state-of-the-art algorithms. Specifically, the proposed algorithm is compared with those in \cite{zhang2020distributed}, \cite{nedic2018improved} and \cite{falsone2023augmented}, which are termed distributed dual gradient tracking algorithm (DDGT), Mirror-P-EXTRA and Augmented Lagrangian Tracking (ALT), respectively.

To illustrate the effectiveness of the proposed algorithm, we employ the IEEE 118-bus system \cite{venkatesh2003comparison}. The cyber network for the IEEE 118-bus system is constructed as a connected undirected graph $(\mathcal{V}, \mathcal{E})$, where $\mathcal{V}=\{1,2, \ldots, 118\}$ represents the set of nodes (corresponding to buses) and the edge set $\mathcal{E}$ is defined by the rule $(i, i+1)$ and $(i, i+2)$ for all $1 \leq i \leq 116$. This means that each node is connected to its nearest neighbor and the next-nearest neighbor, forming a sparse yet connected topology. Based on this topology, we further derive a randomly generated doubly stochastic weight matrix. In this system, generator buses are randomly set to be located at 14 different buses. In the economic dispatch problem, we set the local cost of each generator as a quadratic function $a_i x_i^2 + b_i x_i + c_i$. The coefficients for the local cost functions are adopted from \cite{venkatesh2003comparison}. We set $[\underline{p}_i, \overline{p}_i] = [0,300]$ for each generator. For buses without generators, their corresponding coefficients $a_i$ and $b_i$ are set to zero and their local constraints are set as $\underline{p}_i = \overline{p}_i = 0$. 
The total power demand across the system is given by $ \sum_{i=1}^{14} d_i =950 \mathrm{MW}$, where $d$ is unknown to each bus. Without loss of generality, we set a virtual initial local demand $d_i$ as $(950/14) \mathrm{MW}$ at each bus with a generator and $0 \mathrm{MW}$ at all other buses. 


\begin{figure}[!htb]
    \centering     
    \includegraphics[width=0.4\textwidth]{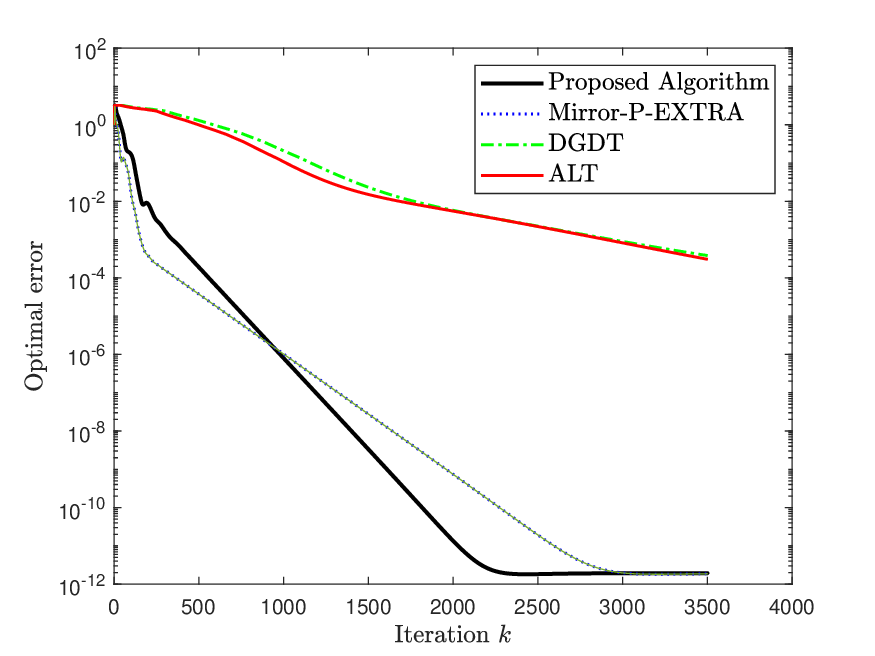}
    \caption{The power mismatch under the five algorithms.}
    \label{fig: power mismatch}
\end{figure}

\begin{figure}[!htb]
    \centering        
    \includegraphics[width=0.4\textwidth]{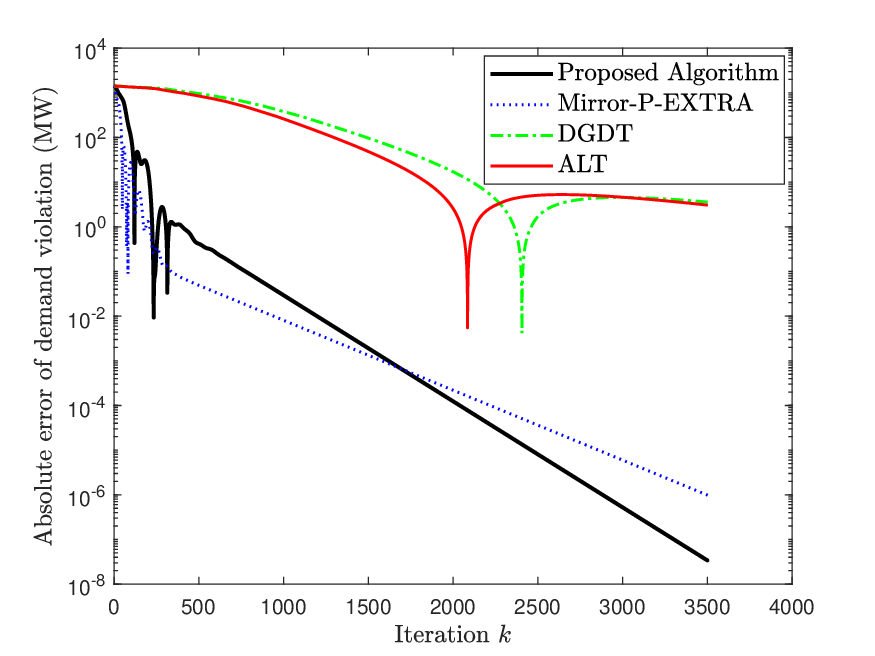}
    \caption{The absolute error of the total power generation and the total power demand under the five algorithms.}
    \label{fig: demand}
\end{figure}
\begin{figure}[!htb]
    \centering        
    \includegraphics[width=0.4\textwidth]{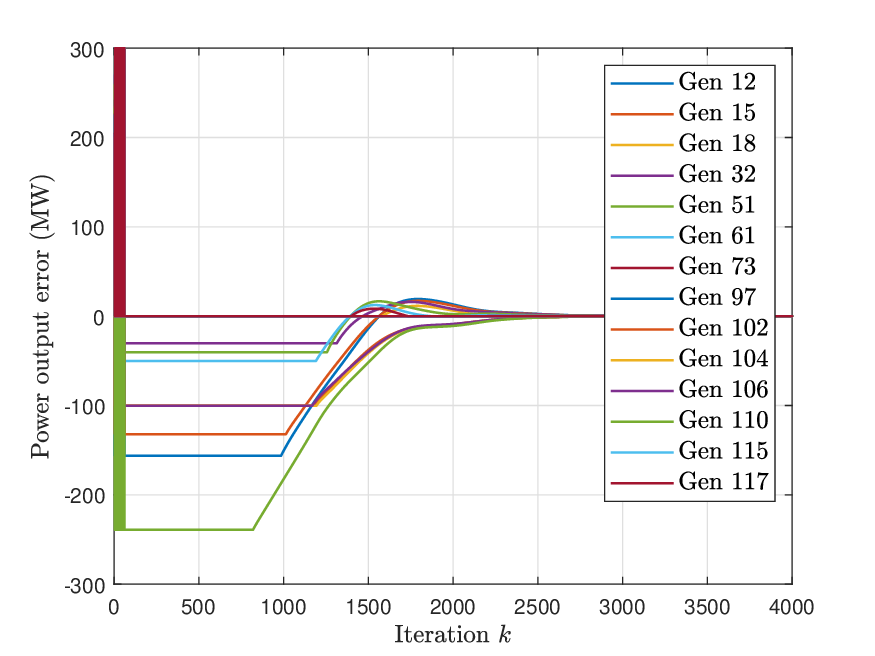}
    \caption{The power output of each generator under the proposed algorithm.}
    \label{fig: power of each}
\end{figure}
\begin{figure}[!htb]
    \centering        
    \includegraphics[width=0.4\textwidth]{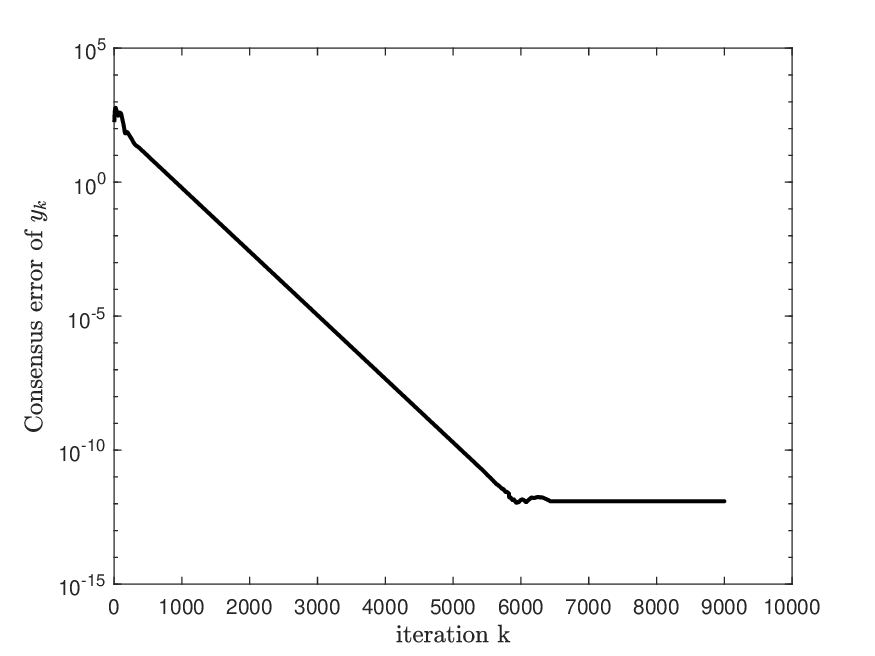}
    \caption{The consensus error of the variable $y_k$ under the proposed algorithm.}
    \label{fig: consensus}
\end{figure}

Figure \ref{fig: power mismatch} plots the total optimal error of the cost function of the IEEE 118-bus system given the total power demand $d$, i.e., $\tfrac{| f(x_{k}) - f(x^*) |}{|f(x_{1}) - f(x^*)|}$ both for the proposed algorithm and for DDGT, Mirror-P-EXTRA, and ALT algorithms. As shown in Fig. \ref{fig: power mismatch}, the convergence of the proposed algorithm is faster than that of all three other algorithms. Since the existence of strong duality, Fig. \ref{fig: power mismatch} also demonstrates the effectiveness of Algorithm \ref{algorithm} for the consensus optimization problem.
Figure \ref{fig: demand} presents the absolute error of the total power generation and the total power demand, i.e., $| \sum_{i=1}^n (x_{i,k} - d_i)|$. As seen, our algorithm again outperforms the others, reaching a violation error below $10^{-6}$ earlier than all other algorithms. 

Figure \ref{fig: power of each} plots the error between the power output of each generator and its optimal power output under the proposed algorithm, i.e., $|x_{i,k} - x_i^*|,~ i\in \mathcal{V} $. It can be seen in Fig. \ref{fig: power of each} that the error of each generator can gradually reach zero. In addition, Fig. \ref{fig: consensus} shows convergence of the consensus error $\| y_{k} - 1_{118}^{\rm{T}} y_{k}\|$ of the variable $y_k$.

\section{Conclusion}\label{sec: Conclusion}
We presented a distributed framework for convex optimization over networks with consensus and global equality couplings. For consensus problems, a linearized augmented Lagrangian scheme achieves non-ergodic $O(1/\sqrt{k})$ rates for both objective error and consensus violation. Via a dual consensus reformulation, the same method applied to economic dispatch yields non-ergodic $O(1/\sqrt{k})$ decay of the objective error and the equality-constraint residual under convexity and Slater’s condition. Experiments on the IEEE 118-bus system confirm faster reduction of objective and feasibility errors than representative state-of-the-art baselines.
\bibliographystyle{IEEEtran}
\bibliography{reference}
\end{document}